\documentclass[11pt]{amsart}
\usepackage{amsthm} 
\usepackage{fancybox}
\usepackage{graphicx}
\usepackage{ascmac}
\usepackage{tikz-qtree}
\usepackage{forest}
\usepackage{mathrsfs}
\usepackage{amscd,verbatim}
\usepackage{floatflt}
\usepackage{amssymb}
\usepackage{tikz-cd}
\usepackage[all]{xy}
\usepackage[colorlinks,linkcolor=blue,citecolor=blue,urlcolor=red]{hyperref}
\usepackage[normalem]{ulem}
\usepackage{geometry}

\usepackage{mathtools}
\newcommand{\perfdg}{\operatorname{perf}_{\operatorname{dg}}}

\renewcommand{\AA}{\mathbb{A}}

\newcommand{\Q}{\mathbb{Q}}
\newcommand{\Z}{\mathbb{Z}}

\newcommand{\sT}{\mathcal{T}}
\newcommand{\sU}{\mathcal{U}}

\newcommand{\sM}{\mathcal{M}}

\newcommand{\bC}{\mathbb{C}}

\newcommand{\HH}{\operatorname{HH}}

\newcommand{\HP}{\operatorname{HP}}

\newcommand{\sX}{\mathcal{X}}

\newcommand{\DM}{\operatorname{\mathbf{DM}}}

\newcommand{\DA}{\operatorname{\mathbf{DA}}}
\newcommand{\SH}{\operatorname{\mathbf{SH}}}
\newcommand{\SHinf}{\operatorname{\mathcal{SH}}}
\newcommand{\Dg}{\operatorname{\mathcal{DG}}}
\newcommand{\SHnc}{\operatorname{\mathbf{SH}}^{\operatorname{nc}}}
\newcommand{\SHncinf}{\operatorname{\mathcal{SH}}^{\operatorname{nc}}}

\newcommand{\dR}{\operatorname{dR}}

\newcommand{\LKl}{{L_{K(1,l)}}}
\newcommand{\Ktop}{K^{\operatorname{top}}}

\newcommand{\perf}{\operatorname{perf}}

\newcommand{\Hom}{\operatorname{Hom}}
\newcommand{\Fun}{\operatorname{Fun}}
\newcommand{\Sp}{\operatorname{\textbf{Sp}}}
\newcommand{\MapncSpQ}{\operatorname{Map}_{\SHncinf_{\Q}}^{\Sp}}
\newcommand{\MapncSp}{\operatorname{Map}_{\SHncinf}^{\Sp}}

\newcommand{\Spec}{\operatorname{Spec}}

\newcommand{\Sm}{\operatorname{\mathbf{Sm}}}

\newcommand{\eff}{{\operatorname{eff}}}

\newcommand{\Nis}{{\operatorname{Nis}}}
\newcommand{\et}{{\operatorname{\acute{e}t}}}

\newcommand{\ch}{{\operatorname{ch}}}

\renewcommand{\lim}{\operatornamewithlimits{\varprojlim}}
\newcommand{\colim}{\operatornamewithlimits{\varinjlim}}

\newcommand{\ccolim}{\operatorname{colim}}
\newcommand{\llim}{\operatorname{lim}}
\newcommand{\ol}{\overline}

\renewcommand{\epsilon}{\varepsilon}

\newcounter{spec}
{\end{list}}%

\newtheorem{lemma}{Lemma}[section]

\newtheorem{thm}[lemma]{Theorem}

\newtheorem{prop}[lemma]{Proposition}
\newtheorem{proposition}[lemma]{Proposition}
\newtheorem{cor}[lemma]{Corollary}

\newtheorem{defn}[lemma]{Definition}

\numberwithin{equation}{section}

\title{A Criterion for Phantomness of dg-categories}
\author{Keiho Matsumoto}

\begin{document}

\begin{abstract}
We study the question of whether the vanishing of additive invariants characterizes phantomness for smooth proper dg categories admitting geometric realizations. More precisely, let $X$ be a smooth proper variety over a field $k$, and let $\sT\subset \perfdg(X)$ be a $k$-linear admissible full dg subcategory. We construct a non-compact motive $\sM(\sT)\in \DM(k,\Q)$ and show that its $l$-adic realization recovers the $K(1,l)$-local algebraic $K$-theory of $\sT$. Analogous statements are obtained for Betti and de Rham realizations, which recover topological $K$-theory and periodic cyclic homology, respectively.

As a consequence, assuming that the Chow motive of $X$ is Kimura-finite, we prove a criterion for phantomness: the vanishing of $L_{K(1,l)}K(\sT_{\overline{k}})_\Q$, of Hochschild homology in characteristic zero, or of rational topological $K$-theory over $\mathbb{C}$ implies that the rational noncommutative motive of $\sT$ vanishes. In this way, our results provide a partial answer to a question raised by Sosna. We also establish a deformation-invariance result for phantomness in smooth proper families.
\end{abstract}

\maketitle


\section{Introduction}\label{sec:intro}
The notion of a \textit{phantom category}, studied in noncommutative algebraic geometry, has become increasingly important due to its connections with the minimal model program and mirror symmetry (see \cite{CKP13}). Throughout, let $k$ be a field. A $k$-linear dg category $\sT$ is said to \textit{admit a geometric realization} if it is a $k$-linear admissible full dg subcategory of $\perfdg(X)$ for some smooth proper $k$-variety $X$.

For a nonzero smooth proper dg category $\sT$ that admits a geometric realization, Gorchinskiy and Orlov \cite{GO13} define $\sT$ to be a \textit{quasi-phantom} if $\HH_*(\sT/k)=0$ and $K_0(\sT)$ is finite. They also define $\sT$ to be a \textit{universal phantom} if
\[
K_0(\sT\otimes Y)=0
\]
for every smooth proper $k$-variety $Y$. In this paper we introduce the following closely related notion.

\begin{defn}
We call $\sT$ a \textit{motivic quasi-phantom} if
\[
K_0(\sT\otimes Y)_\Q=0
\]
for every smooth proper $k$-variety $Y$.
\end{defn}

If $\sT$ admits a geometric realization, then $\sT$ is a motivic quasi-phantom if and only if the rational noncommutative motive of $\sT$ vanishes. When the base field $k$ is $\bC$ (or, more generally, an algebraically closed field of infinite transcendence degree over its prime subfield), one has
\[
\text{quasi-phantom}\ \Longrightarrow\ \text{motivic quasi-phantom}
\]
(see \cite[Lemma 5.3]{GO13}). Conversely, if $\sT$ is an admissible full dg subcategory of $\perfdg(X)$ and $K_0(X)_{\mathrm{tor}}$ is finitely generated (for instance, this holds under the Bass conjecture when $k$ is a number field or a finite field), then
\[
\text{motivic quasi-phantom}\ \Longrightarrow\ \text{quasi-phantom}.
\]
These implications can be summarized as follows:
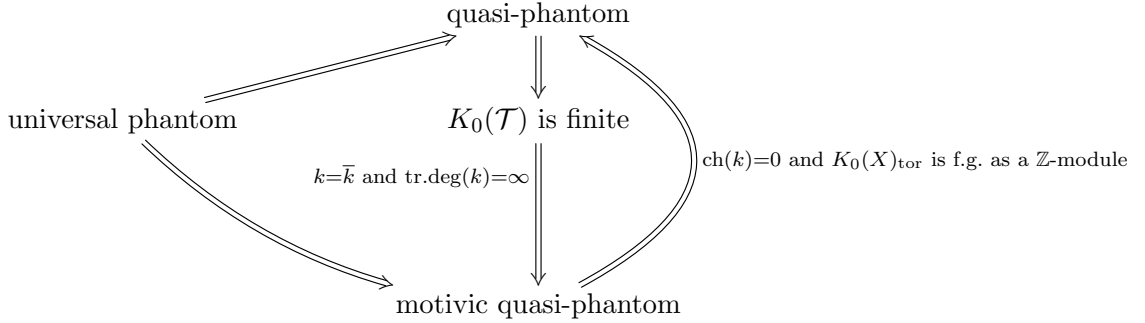
\[
\resizebox{\textwidth}{!}{$
\xymatrix{
&& \text{quasi-phantom}\ar@{=>}[d] \\
\text{universal phantom} \ar@{=>}[rru]\ar@{=>}@/_18pt/[rrdd] & &K_0(\sT)\text{ is finite} \ar@{=>}[dd]_(.3){k=\ol{k} \text{ and } \mathrm{tr.deg}(k)=\infty} \\
&& \\
&&\text{motivic quasi-phantom}\ar@{=>}@/_60pt/[uuu]_-{\ch(k)=0 \text{ and }K_0(X)_{\mathrm{tor}} \text{ is f.g.\ as a }\Z\text{-module}} 
}
$}
\]

If $\sT$ is a motivic quasi-phantom, then every additive invariant with rational coefficients vanishes. In particular, for any prime $l$ invertible in $k$, one has $\LKl K(\sT_{\ol{k}})_\Q=0$. Furthermore, if $\ch(k)=0$, then $\HH_*(\sT/k)=0$. The main result of this paper is that this converse implication holds. It may also be viewed as providing a partial answer to a question raised by Sosna \cite[Question~1]{S20}.
\begin{thm}\label{criterionphantom}
Let $k$ be a field, and let $\sT$ be a non-zero $k$-linear smooth proper dg category admitting a geometric realization $\sT \hookrightarrow \perfdg(X)$. Assume that $X$ has Kimura-finite Chow motive. Then:
\begin{itemize}
\item[(1)] Let $l$ be a prime such that $l\in k^*$. If $\LKl K(\sT_{\ol{k}})_\Q=0$, then $\sT$ is a motivic quasi-phantom.
\item[(2)] Assume $\ch(k)=0$. If $\HH_*(\sT/k)=0$, then $\sT$ is a motivic quasi-phantom.
\item[(3)] Assume $k=\bC$. If $\Ktop(\sT)_{\Q}=0$, then $\sT$ is a motivic quasi-phantom.
\end{itemize}
\end{thm}

Here the notion of \textit{Kimura finiteness} for Chow motives plays a central role. Kimura finiteness is a property of Chow motives introduced independently by Kimura \cite{K05} and O'Sullivan. They conjectured that every Chow motive is Kimura-finite. 
For smooth projective surfaces with vanishing geometric genus, the Kimura-finiteness of the Chow motive is equivalent to Bloch's conjecture.

As a corollary of this theorem, we can establish the deformation invariance of the property of admitting a phantom category, as follows. 

\begin{cor} Let $S$ be a Noetherian connected scheme, and let $f:\sX \to S$ be a smooth proper morphism. Assume that the Chow motive of each fiber is Kimura-finite. Suppose that for some point $0\in S$, the fiber $\sX_0$ admits a motivic quasi-phantom category $\sT_0$, and that there exists a prime number $l\in \kappa(0)^*$ such that $L_{K(1,l)}K(\sT_0)\neq 0$. Then there exists an \'etale morphism $\sU \to S$ whose image contains $0$ and such that every fiber of $\sX_{\sU} \to \sU$ admits a motivic quasi-phantom category. Moreover, under the Belmans--Okawa--Ricolfi conjecture \cite[Conjecture 9.9]{BOR20}, the nonvanishing condition $L_{K(1,l)}K(\sT_0)\neq 0$ is unnecessary.
\end{cor}

This corollary may be viewed as a generalization of the argument of B\"ohning--Bothmer--Katzarkov--Sosna in \cite[Section~10]{BBKS15}: for families of $\Z/5\Z$-torsion numerical Godeaux surfaces, they prove a corresponding deformation-invariance statement by combining Bloch's conjecture with the deformation invariance of the Grothendieck group. In our setting, the Kimura-finiteness assumption plays the role that allows one to extend this strategy to a broader class of families. 

\subsection{Notation}
For a commutative ring $R$, we use the following notation.

\begin{itemize} 
    \item[] $\Sm(R)$ denotes the category of smooth varieties over $R$. 
    \item[] $\Dg(R)$ denotes the $\infty$-category of idempotent-complete dg categories of finite type over $R$, introduced by To\"en--Vaqui\'e.
    \item[] $\DM(R)$ denotes the triangulated category of mixed motives: for fields $R$, this is the category constructed by Voevodsky \cite{V00b}, while for commutative rings $R$, we use the construction of Cisinski--D\'eglise \cite{CD19}
    \item[] $\SHncinf(R)$ denotes the stable $\infty$-category of noncommutative motivic spectra over $R$ (see \cite{R12}).
    \item[] $\SHnc(R)$ denotes the homotopy category of $\SHncinf(R)$.
    \item[] For an algebra object $A$ in $\SHncinf(R)$, $\SHncinf(R)_A$ denotes the $\infty$-category of $A$-modules in $\SHncinf(R)$.
    \item[] $\SHnc(R)_A$ denotes the homotopy category of $\SHncinf(R)_A$.
\end{itemize}

\section*{Acknowledgments}
The author is grateful to Atsushi Takahashi and Shinnosuke Okawa for valuable discussions on noncommutative algebraic geometry and phantom phenomena. The author would also like to thank Takumi Otani for helpful comments on an earlier draft. This work was supported by JSPS KAKENHI Grant Numbers 22K13898 and 25KJ0210.

\section{Dg-categories and motives}\label{MotivesandNon-commutative motives}
In this section, we fix a field $k$ and fix a prime number $l$ with $l\in k^*$. We begin by recalling some basic facts from motivic homotopy theory that will be used later. We write $H\Q$ for the motivic rational Eilenberg--MacLane spectrum in $\SH(k)$. There is an equivalence
\[
\DM(k,\Q) \simeq \SH(k)_{H\Q}.
\]
For fields of characteristic $0$, this is proved in \cite{RO08}; for arbitrary fields, see \cite[Theorem 5.8]{HSO17}. Denote by $KGL_\Q$ the rational algebraic K-theory spectrum in the motivic stable homotopy category $\SH(k)$. By \cite[Corollary 14.2.17]{CD19}, one has an isomorphism of $H\Q$-algebra
\begin{equation}\label{KGLQHQ}
KGL_\Q \simeq \bigoplus_{m \in \Z} H\Q(m)[2m]
\end{equation}
in $\SH(k)$. 

In \cite[Corollary~1.16]{R13}, Robalo proves that there is a fully faithful embedding $L_{KGL_\Q}:\SHinf(k)_{KGL_\Q} \hookrightarrow \SHncinf(k)_\Q$. Although the argument in \cite{R13} is given under the assumption of resolution of singularities, the same strategy applies without this assumption over $\Q$-coefficients, because $\SH(k)_{H\Q}$ is compactly generated by objects of the form $X_+(q)$, where $X$ is smooth and proper and $q\in\Z$ (see \cite{R05}). The embedding $L_{KGL_\Q}:\SHinf(k)_{KGL_\Q} \hookrightarrow \SHncinf(k)_\Q$ admits the right adjoint $\sM_{KGL_\Q}:\SHncinf(k)_\Q \to \SHinf(k)_{KGL_\Q}$ (see \cite[Corollary 1.13]{R13}).

\subsection{commutative motive of a dg-category} Now consider the following diagram of triangulated categories:
\[
\xymatrix{
\DM_\Q\ar[r]^-{\simeq}_-{H} &\SH_{H\Q} \ar@/^13pt/[r]^{-\otimes_{H\Q}KGL_{\Q}} \ar@{}[r]|{\perp} & \SH_{KGL_\Q} \ar@/^13pt/[l]^{I} \ar@/^13pt/[r]^{L_{KGL_{\Q}}} \ar@{}[r]|{\perp}  & \ar@/^13pt/[l]^{\sM_{KGL_{\Q}}} 
} \SHnc_\Q
\]
Here, $I$ is thje forgetful functor. We may construct a $\Q$-linear functor
\[
\sM=H^{-1}\circ I \circ \sM_{KGL_{\Q}}.
\]
For a smooth proper $k$-variety $X$ and a $k$-linear admissible full dg subcategory $\sT\subset \perfdg(X)$, this yields a motive
\[
\sM(\sT)=H^{-1}\circ I \circ \sM_{KGL_{\Q}} \bigl(\sU(\sT) \bigr).
\]
We study basic properties of this motive.

\subsection{Realization of $\sM(\sT)$}From now on, we assume that $k$ is algebraically closed. Now consider the $l$-adic realization 
\[
    R_l:\DM(k,\Q) \to D(\Q_l)^{op}.
\]
In this subsection we prove the following theorem.

\begin{thm}\label{ladicrealizationthm}
Let $\sT$ be a smooth proper dg category admitting a geometric realization. The $l$-adic \'{e}tale realization of $\sM(\sT)$ satisfies
\[
H^i \Bigl(R_l\bigl( \sM(\sT) \bigr) \Bigr) \simeq \pi_i\LKl K(\sT)
\]
\end{thm}

We prove this theorem by constructing two symmetric monoidal functors
\[
F_1,F_2:\SH_{KGL_\Q} \to D(\Q_l[\beta^{\pm}]\text{-}\mathrm{mod})^{op},
\]
and showing that they are equivalent.

\subsection{Construction of $F_1$}
Consider the composite functor
\begin{equation}\label{IHRl}
\SH_{KGL_{\Q}}\overset{I}{\to} \SH_{H\Q} \overset{\simeq}{ \underset{H}{\leftarrow}} \DM_\Q \overset{R_l}{\to} D(\Q_l)^{op}.
\end{equation}
The functor $R_l\circ H^{-1}$ sends $KGL_\Q \simeq \bigoplus_{m \in \Z} H\Q(m)[2m]$ to $\Q_l[\beta^\pm]$. Here $\beta$ is a degree-$2$ generator. Hence \eqref{IHRl} factors through a monoidal functor
\[
\xymatrix{
\SH_{KGL_\Q} \ar@{.>}[r]^-{F_1} \ar[d]_-{I}& D(\Q_l[\beta^\pm]\text{-mod})^{op} \ar[d]^-{\text{forgetful}} \\
\SH_{H\Q} \ar[r]_-{R_l\circ H^{-1}} & D(\Q_l)^{op}.
}
\]

\begin{lemma}
The functor $F_1$ admits a right adjoint.
\end{lemma}

\begin{proof}
Since $\SH_{H\Q}$ is compactly generated, it follows from \cite[Proposition 3.28]{R12} that $\SH_{KHL_\Q}$ is also compactly generated. By the adjoint functor theorem \cite[Theorem 1.4.4.1]{L18}, it suffices to show that $F_1$ preserves colimits. By \cite[Corollary 3.4.4.6]{L18}, the functor $I$ preserves colimits, hence so does $R_l\circ H^{-1}\circ I$. Since the forgetful functor $D(\Q_l[\beta^{\pm}]\text{-mod}) \to D(\Q_l)$ is conservative and preserves colimits, it follows that $F_1$ preserves colimits.
\end{proof}

\subsection{Construction of $F_2$} Let $\Sp$ denote the $\infty$-category of spectra, and let $\widehat{\Sp}$ denote the big $\infty$-category of spectra. Let
\[
K^S:\Dg(k)\longrightarrow \Sp
\]
denote the nonconnective $K$-theory of dg categories constructed in \cite{CT11}.
Let $1_{nc}$ be the unit object of $\SHnc$. Consider the functor
\[
(\LKl\, l_{\AA^1}^{nc}K^S(-))_\Q:\Dg(k)\longrightarrow \Sp_\Q.
\]
By definition, this functor sends Nisnevich squares of dg categories to pullback--pushout squares in $\widehat{\Sp_\Q}$ and sends the zero dg category to the zero object. It is also $\AA^1$-invariant. Hence we may regard
\[
(\LKl\, l_{\AA^1}^{nc}K^S(-))_\Q\in \SHncinf_\Q.
\]
We then consider the associated representing functor
\[
\MapncSpQ\!\bigl(-,(\LKl\, l_{\AA^1}^{nc}K^S)_\Q\bigr):\SHncinf_\Q\longrightarrow \Sp_\Q^{op}.
\]
Composing with the enriched Yoneda embedding, the composite
\[
\Dg(k)^{op}\xrightarrow{\ \sU\ }\SHncinf_\Q\longrightarrow \Sp_\Q^{op}
\]
sends a dg category $\sT$ to $(\LKl\, l_{\AA^1}^{nc}K^S(\sT))_\Q$.

Now let $\sT$ admit a geometric realization $i_{\sT}:\sT\hookrightarrow \perfdg(X)$.
Let $l_{\sT}:\perfdg(X)\to \sT$ denote the left adjoint of the admissible fully faithful functor $i_{\sT}$, and consider the composite
\[
e_\sT:\perfdg(X)\xrightarrow{\,l_\sT\,}\sT\xrightarrow{\,i_\sT\,}\perfdg(X).
\]
Since $e_{\sT}$ is idempotent, in $\SHncinf$ one has an identification
\[
\sU(\sT)\simeq \ccolim\bigl(\sU(\perfdg(X)),\,\sU(e_{\sT}^{op})\bigr).
\]

\begin{lemma}
In $\Sp$ there is an equivalence
\[
K(\sT)\simeq l_{\AA^1}^{nc}K^S(\sT).
\]
In particular, in $\Sp_\Q$ one has
\[
(\LKl\, l_{\AA^1}^{nc}K^S(\sT))_\Q \simeq \bigl(\LKl\, K(\sT)\bigr)_\Q.
\]
\end{lemma}

\begin{proof}
From the adjunction
\[
\xymatrix{
\Fun(\Dg,\Sp)\ar@/^13pt/[rr]^{l_{\AA^1}^{nc}} \ar@{}[rr]|{\perp} && \Fun_{\Nis,\AA^1}(\Dg,\Sp)\ar@/^13pt/[ll]
}
\]
and the induced unit map, we obtain a commutative diagram in $\Sp$:
\[
\xymatrix{
K(X) \ar[rr]^{K(e_\sT)} \ar[d] && K(X)  \ar[d] \\
l_{\AA^1}^{nc}K^S(X) \ar[rr]_{l_{\AA^1}^{nc}K^S(e_{\sT})}& &l_{\AA^1}^{nc}K^S(X).
}
\]
Since $X$ is smooth, both vertical maps are equivalences in $\Sp$ (see \cite[Corollary 1.14]{R13}). As $e_{\sT}$ is idempotent, we obtain an equivalence
\[
K(\sT)\simeq \llim\bigl(K(X),\,K(e_{\sT})\bigr)
\]
in $\Sp$, and hence
\[
K(\sT)\simeq \ccolim\bigl(K(X),\,K(e_{\sT}^{op})\bigr)
\]
in $\Sp^{op}$.
On the other hand, from $\sU(\sT)\simeq \ccolim\bigl(\sU(\perfdg(X)),\sU(e_{\sT}^{op})\bigr)$ we deduce an equivalence in $\Sp^{op}$
\begin{eqnarray*}
l_{\AA^1}^{nc}K^S(\sT)
&\simeq&
\MapncSp\!\bigl(\ccolim(\sU(\perf(X)),\sU(e_{\sT}^{op})),\,1_{nc}\bigr)\\
&\simeq&
\ccolim\bigl(l_{\AA^1}^{nc}K(X),\,l_{\AA^1}^{nc}K(e_{\sT}^{op})\bigr).
\end{eqnarray*}
This proves the claim.
\end{proof}

Now consider the composite functor
\begin{equation}\label{compositionfunctorF2}
\SHinf_{KGL_\Q}\xrightarrow{\,L_{KGL_\Q}\,}\SHncinf_\Q
\xrightarrow{\ \MapncSpQ\!\bigl(-,(\LKl\, l_{\AA^1}^{nc}K^S)_\Q\bigr)\ }
\Sp_\Q^{op}.
\end{equation}

\begin{lemma}
The homotopy category functor induced by \eqref{compositionfunctorF2} factors through a functor
\[
\xymatrix{
\SH_{KGL_\Q} \ar[d] \ar@{.>}[rr]_-{F_2} && D(\Q_l[\beta^{\pm}]\text{-mod})^{op}   \ar[d]^{\operatorname{forgetful}}\\
 \ar[rr]\SHnc_\Q  && \Sp_\Q^{op},
}
\]
and $F_2$ is symmetric monoidal and admits a right adjoint.
\end{lemma}

\begin{proof}
The functor \eqref{compositionfunctorF2} sends $KGL_\Q$ to $\LKl\,K(k)$; the latter carries a canonical $E_\infty$-algebra structure and identifies with $\Q_l[\beta^{\pm}]$. Moreover, for a smooth $k$-scheme $X$, it sends $M(X)\otimes_{H\Q}KGL_\Q$ to $\LKl\,K(X)$, which is naturally a module over $\LKl\,K(k)$. Since $\SH_{KGL_\Q}$ is compactly generated by the objects $\{\,M(X)\otimes_{H\Q}KGL_\Q \mid X \text{ smooth}\,\}$ and \eqref{compositionfunctorF2} preserves colimits, it follows that for every $x\in \SH_{KGL_\Q}$ the object \eqref{compositionfunctorF2}$(x)$ carries a canonical $\LKl\,K(k)$-module structure. Therefore \eqref{compositionfunctorF2} factors through a functor
\[
F_2:\SH_{KGL_\Q}\longrightarrow D(\Q_l[\beta^{\pm}]\text{-mod})^{op}.
\]
By \cite[Corollary 3.4.4.6]{L18}, the forgetful functor is conservative and preserves colimits. It follows that $F_2$ preserves colimits as well. For every smooth variety $X$ over $k$, Thomason's theorem \cite{T85} yields a functorial isomoprhism
\[
\pi_j \LKl K(X)[1/l] \simeq \bigoplus_{i\in \mathbb Z} H^{2i-j}_{\et}(X,\mathbb Z_l(i))[1/l][2i].
\]
Consequently, for smooth varieties $X$ and $Y$, we obtain the K\"unneth formula for $K(1,l)$-local $K$-theory:
\[
\LKl K(X)[1/l] \otimes_{\LKl K(k)[1/l]} \LKl K(Y)[1/l]
\simeq
\LKl K(X\times_k Y)[1/l].
\]
Therefore, the natural lax symmetric monoidal structure on $F_2$ is an equivalence on smooth varieties. Since $F_2$ preserves colimits, it follows that $F_2$ is symmetric monoidal.\end{proof}

\subsection{Equivalence of $F_1$ and $F_2$} Denote by $a$ the functor $-\otimes_{H\Q} KGL_\Q: \SH_{H\Q} \to \SH_{KGL_\Q}$.
\begin{proposition}
Functors $F_1\circ a $ and $F_2 \circ a$ are equivalence.
\end{proposition}
\begin{proof}
The category $\DM(k,\Q)$ is obtained from $\DM^{\eff}(k,\Q)$ by inverting the Tate motive $\mathbf{Q}(1)$, or equivalently by stabilizing with respect to Tate twist. In particular, there is a canonical fully faithful functor
\[
i:\DM^{\eff}(k,\Q) \hookrightarrow \DM(k,\Q),
\]
and the Tate twist becomes an autoequivalence on $\DM(k,\Q)$ (see \cite{V10}). Since $F_1\circ a$ and $F_2\circ a$ are symmetric monoidal, it suffices to show that $F_1\circ a\circ i$ and $F_2\circ a\circ i$ are equivalent. By \cite[Theorem B.13]{A14}, there is an equivalence
\[
    \DA^{\et,\eff}(k,\Q) \simeq \DM^{\eff}(k,\Q).
\]
Here, $\DA^{\et,\eff}(k,\Q)$ is obtained from the category of complexes $\mathrm{Cpl}(\mathrm{Shv}_{\et}(\Sm(k),\Q))$ by localizing at $\AA^1$-homotopy equivalences. By \cite[Theorem B.13]{A14}, $\DA^{\et,\eff}(k,\Q)$ is compactly generated by $\Q(X)[i]$ for smooth variety $X$ and $i\in\Z$. Since $F_1$ and $F_2$ preserves colimit, it is suffices to show that 
\[
 \Sm(k) \to \DA^{\et,\eff}(k,\Q) \simeq \DM^\eff(k,\Q) \hookrightarrow \DM(k,\Q) \overset{F_1\circ a}{\underset{F_2\circ a}{\rightrightarrows}} D(\Q_l[\beta^{\pm}]\text{-}\mathrm{mod})^{op}
\]
are equivalence. 

For a smooth variety $X$, we compare the two functors via Thomason's \'etale descent description of $K(1)$-local algebraic $K$-theory \cite{T85}.
More precisely, for a prime $l$ invertible in $k$, there is a strongly convergent spectral sequence
\[
E_2^{p,2q}=H_{\et}^{p}(X,\Z_l(q)) \Longrightarrow \pi_{2q-p}(\LKl K(X)),
\]
with $E_2^{p,2q+1}=0$.
After tensoring with $\Q_l$, this spectral sequence is degenerate. Thus there is an isomorphism $F_1\circ a(M(X)_\Q) \simeq F_2\circ a(M(X)_\Q)$. This spectral sequence is obtained by applying $l$-adic completion and $K(1)$-localization to the \'etale sheafification of the nonconnective $K$-theory presheaf, and then using the Postnikov filtration of the resulting \'etale-local spectrum. 
By construction, this spectral sequence is functorial in smooth varieties. Therefore, the restrictions of $F_1$ and $F_2$ to $\Sm(k)$ are equivalent. 
\end{proof}

    \begin{proposition}
The functors $F_1$ and $F_2$ are equivalent.
\end{proposition}

\begin{proof}
Let $G_1$ and $G_2$ be right adjoints of $F_1$ and $F_2$, respectively.

We first note that for any $A,B\in \SH_{H\Q}$, one has
\begin{align*}
\Hom_{\SH_{KGL_\Q}}(aA,aB)
&\simeq \Hom_{\SH_{H\Q}}(A,I(aB)) \\
&\simeq \Hom_{\SH_{H\Q}}(A,B\otimes_{H\Q}KGL_\Q) \\
&\simeq \Hom_{\SH_{H\Q}}\Bigl(A,\bigoplus_{i\in\mathbb Z} B(i)[2i]\Bigr) \\
&\simeq \bigoplus_{i\in\mathbb Z}\Hom_{\SH_{H\Q}}(A,B(i)[2i]),
\end{align*}
where we used \eqref{KGLQHQ}. In particular, if $\mathcal G$ is a set of compact generators of $\SH_{H\Q}$, then $a(\mathcal G)$ is a set of compact generators of $\SH_{KGL_\Q}$.

By the previous proposition, we have a natural equivalence
\[
F_1\circ a \simeq F_2\circ a.
\]
Hence, for every $C\in \mathcal G$ and every object $M\in D(\Q_l[\beta^{\pm}]\text{-}\mathrm{mod})^{op}$, we obtain
\begin{align*}
\Hom_{\SH_{KGL_\Q}}(aC,G_1(M))
&\simeq \Hom_{D(\Q_l[\beta^{\pm}]\text{-}\mathrm{mod})^{op}}(F_1(aC),M) \\
&\simeq \Hom_{D(\Q_l[\beta^{\pm}]\text{-}\mathrm{mod})^{op}}(F_2(aC),M) \\
&\simeq \Hom_{\SH_{KGL_\Q}}(aC,G_2(M)).
\end{align*}
These isomorphisms are functorial in both $C$ and $M$.

Since $a(\mathcal G)$ compactly generates $\SH_{KGL_\Q}$, an object of $\SH_{KGL_\Q}$ is determined by its Hom-groups from the objects $aC$ with $C\in\mathcal G$. Therefore, for every $M$, the above functorial isomorphisms imply
\[
G_1(M)\simeq G_2(M).
\]
Thus the right adjoints are naturally equivalent:
\[
G_1\simeq G_2.
\]
It follows that their left adjoints are naturally equivalent as well. Hence
\[
F_1\simeq F_2.
\]
This proves the proposition.
\end{proof}

\subsection{Realizations of $\sM(\sT)$} First, we prove Theorem~\ref{ladicrealizationthm}.
\begin{proof}[Proof of Theorem~\ref{ladicrealizationthm}]
Set
\[
E_{\sT}:=\sM_{KGL_\Q}\bigl(\sU(\sT)\bigr)\in \SH_{KGL_\Q}.
\]
By definition,
\[
\sM(\sT)=H^{-1}\circ I(E_{\sT}).
\]

We now evaluate the two functors $F_1$ and $F_2$ on $E_{\sT}$.

First, by the definition of $F_1$, after applying the forgetful functor
\[
D(\Q_l[\beta^{\pm}]\text{-}\mathrm{mod})^{op}\longrightarrow D(\Q_l)^{op},
\]
we obtain
\[
F_1(E_{\sT}) \simeq R_l\bigl(H^{-1}I(E_{\sT})\bigr)
= R_l\bigl(\sM(\sT)\bigr).
\]

Next, since $\sT$ admits a geometric realization, $\sU(\sT)$ belongs to the essential image of the fully faithful functor
\[
L_{KGL_\Q}:\SH_{KGL_\Q}\hookrightarrow \SHnc_\Q.
\]
Hence the counit
\[
L_{KGL_\Q}(E_{\sT})=L_{KGL_\Q}\sM_{KGL_\Q}\bigl(\sU(\sT)\bigr)\longrightarrow \sU(\sT)
\]
is an equivalence. Therefore, by the definition of $F_2$, we have
\begin{align*}
F_2(E_{\sT})
&\simeq
\MapncSpQ\!\bigl(L_{KGL_\Q}(E_{\sT}),(\LKl\, l_{\AA^1}^{nc}K^S)_\Q\bigr) \\
&\simeq
\MapncSpQ\!\bigl(\sU(\sT),(\LKl\, l_{\AA^1}^{nc}K^S)_\Q\bigr)\\
&\simeq (\LKl K(\sT))_\Q.
\end{align*}

On the other hand, by the previous proposition, the functors $F_1$ and $F_2$ are naturally equivalent. Evaluating this equivalence at $E_{\sT}$, we obtain
\[
R_l\bigl(\sM(\sT)\bigr)\simeq (\LKl K(\sT))_\Q
\]
in $D(\Q_l)^{op}$.

Taking the $i$-th cohomology group on the left and the $i$-th homotopy group on the right, we conclude that
\[
H^i\Bigl(R_l\bigl(\sM(\sT)\bigr)\Bigr)\simeq \pi_i\LKl(\sT).
\]
This proves the theorem.
\end{proof}

Similarly, the corresponding theorem also holds for Betti realization and de Rham realization.
\begin{thm}
Let $\sT$ be a smooth proper $\mathbb{C}$-linear dg category (resp. a smooth proper $k$-linear dg category, where $\operatorname{char}(k)=0$ and $k$ is not assumed to be algebraically closed) admitting a geometric realization. Then the Betti realization (resp. the de Rham realization) of $\sM(\sT)$ satisfies
$$
H^i\Bigl(R_B\bigl(\sM(\sT)\bigr)\Bigr)\simeq \pi_i\Ktop(\sT)
\qquad
\Bigl(\text{resp. } H^i\Bigl(R_{\mathrm{dR}}\bigl(\sM(\sT)\bigr)\Bigr)\simeq \pi_i\HP(\sT)\Bigr).
$$
\end{thm}
\begin{proof}
The proof is obtained by the same argument as in the $l$-adic case, replacing the $l$-adic \'etale realization functor by the Betti realization functor $\DM(\bC,\Q) \to D(\Q)^{op}$ (resp. by the de Rham realization functor $\DM(k,\Q) \to D(k)^{op}$), and replacing $\LKl l^{nc}_{\AA^1}K^S$ by $\Ktop$ (resp. by $\HP$). One then uses Atiyah--Hirzebruch's result \cite{AH61}
\[
\pi_j\Ktop(X^{an})_{\Q}\simeq \bigoplus_{i\in\Z} H^{j+2i}(X^{an},\Q),
\]
(resp. Connes, Feigin--Tsygan's theorem \cite{C83}, \cite{FT87}
\[
\HP_j(X/k) \simeq \bigoplus_{i\in\Z} H^{j+2i}_{\dR}(X/k))
\]
 to conclude by the same formal argument.
\end{proof}
\section{Phantom and motives}In this section, we fix a field $k$ and fix a prime number $l$ with $l\in k^*$.
Let $\sT$ be a smooth proper dg-category which admits a geometrical realization 
\[
    \sT \hookrightarrow \perf(X).
    \]
\begin{lemma}\label{lem3.1}
The motive $\sM(\sT) \in \DM(k,\Q)$ is zero if and only if $\sT$ is motivic quasi-phantom.
\end{lemma}
\begin{proof}
Now
\[
\sM(\sT)=H^{-1}\circ I\circ \sM_{KGL_\Q}(\sU(\sT))
\]
is a direct summand of $\bigoplus_{i\in\Z}M(X)_\Q(i)[2i]$.

Let $Y$ be a smooth projective variety over $k$. By adjunction and the definition of
$\sM(\sT)$, we have
\begin{align*}
\Hom_{\SH_{KGL_\Q}}\bigl(M(Y)_\Q\otimes KGL_\Q,\sM_{KGL_\Q}(\sU(\sT))\bigr)
&\simeq
\Hom_{\DM(k,\Q)}\bigl(M(Y)_\Q,\sM(\sT)\bigr).
\end{align*}
On the other hand, by Kontsevich's formula for mapping spectra in noncommutative motives, proved by Robalo \cite[Corollary 1.13]{R13}, the left-hand side identifies with
\[
K_0(\perfdg(Y)^{op}\otimes_k \sT)_\Q.
\]
Since $Y$ is smooth and projective, $\perfdg(Y)^{op}$ is Morita equivalent to $\perfdg(Y)$.
Therefore
\[
K_0(\sT\otimes_k \perfdg(Y))_\Q
\simeq
\Hom_{\SH_{KGL_\Q}}\bigl(M(Y)_\Q\otimes KGL_\Q,\sM_{KGL_\Q}(\sU(\sT))\bigr).
\]
Using the decomposition
\[
KGL_\Q \simeq \bigoplus_{n\in\mathbb Z} H\Q(n)[2n],
\]
we obtain
\[
M(Y)_\Q\otimes KGL_\Q
\simeq
\bigoplus_{n\in\mathbb Z} M(Y)_\Q(n)[2n].
\]
Hence
\[
K_0(\sT\otimes_k \perfdg(Y))_\Q
\simeq
\bigoplus_{n\in\mathbb Z}
\Hom_{\DM(k,\Q)}\bigl(M(Y)_\Q(n)[2n],\sM(\sT)\bigr).
\]

If $\sM(\sT)=0$, then clearly
\[
K_0(\sT\otimes_k \perfdg(Y))_\Q=0
\]
for every smooth projective $Y$.

Conversely, assume that
\[
K_0(\sT\otimes_k \perfdg(X))_\Q=0.
\]
Then we have
\[
\Hom_{\DM(k,\Q)}\bigl(M(X)_\Q(n)[2n],\sM(\sT)\bigr)=0.
\]
Since $\sM(\sT)$ is a direct summand of $\bigoplus_{i\in\Z}M(X)_\Q(i)[2i]$, it follows
from Yoneda's lemma that
\[
\sM(\sT)=0.
\]
This proves the proposition.
\end{proof}
\begin{lemma}\label{exlem}
    Given a field extension $L/k$. Then $\sT$ is motivic quasi-phantom if and only if $\sT_{L}$ is motivic quasi-phantom.
\end{lemma}
\begin{proof}
Let
\[
i_{\sT}:\sT\hookrightarrow \perfdg(X)
\]
be an admissible embedding, and let
\[
e_{\sT}:\perfdg(X)\to \perfdg(X)
\]
be the associated idempotent endofunctor. As explained above, in $\SHnc_\Q$ one has
\[
\sU(\sT)\simeq \colim\bigl(\sU(\perfdg(X)),\sU(e_{\sT}^{op})\bigr).
\]
Applying the functor
\[
H^{-1}\circ I\circ \sM_{KGL_\Q}
\]
and using
\[
I\bigl(M(X)_\Q\otimes KGL_\Q\bigr)
\simeq
M(X)_\Q\otimes KGL_\Q
\simeq
\bigoplus_{i\in\mathbb Z} M(X)_\Q(i)[2i],
\]
we obtain an idempotent endomorphism
\[
f:\bigoplus_{i\in\mathbb Z} M(X)_\Q(i)[2i]
\longrightarrow
\bigoplus_{i\in\mathbb Z} M(X)_\Q(i)[2i]
\]
such that
\[
\sM(\sT)\simeq
\colim\Bigl(\bigoplus_{i\in\mathbb Z} M(X)_\Q(i)[2i],\,f\Bigr).
\]

After base change to $L$, the admissible subcategory $\sT_L\subset \perfdg(X_L)$ is defined by the base-changed idempotent $(e_{\sT})_L$. Hence the same construction yields
\[
\sM(\sT_L)\simeq
\colim\Bigl(\bigoplus_{i\in\mathbb Z} M(X_L)_\Q(i)[2i],\,f_L\Bigr),
\]
where $f_L$ denotes the image of $f$ under the base-change functor
\[
\DM(k,\Q)\to \DM(L,\Q).
\]
By \cite[Theorem 14.3.3 and Proposition 4.2.4]{CD19}, the functor is conservative and preserves colimit. Thus, by Lemma~\ref{lem3.1}, we obtain the claim.
\end{proof}
\begin{thm}\label{criterionphantominsection3}
Let $k$ be a field, and let $\sT$ be a $k$-linear smooth proper dg category admitting a geometric realization $\sT \hookrightarrow \perfdg(X)$. Assume that $X$ has Kimura-finite Chow motive. Then:
\begin{itemize}
\item[(1)] Let $l$ be a prime such that $l\in k^*$. If $\LKl K(\sT_{\ol{k}})_\Q=0$, then $\sT$ is a motivic quasi-phantom.
\item[(2)] Assume $\ch(k)=0$. If $\HH_*(\sT/k)=0$, then $\sT$ is a motivic quasi-phantom.
\item[(3)] Assume $k=\bC$. If $\Ktop(\sT)_{\Q}=0$, then $\sT$ is a motivic quasi-phantom.
\end{itemize}
\end{thm}
\begin{proof}
    We shall prove only \textup{(1)}, as the proofs of \textup{(2)} and \textup{(3)} are entirely analogous. By Lemma~\ref{exlem}, we can assume $k$ is algebraically closed. 
    
    As explained above, there exists an idempotent endomorphism
\[
f:\bigoplus_{i\in\mathbb Z} M(X)_\Q(i)[2i]\longrightarrow \bigoplus_{i\in\mathbb Z} M(X)_\Q(i)[2i]
\]
such that
\[
\sM(\sT)\simeq \operatorname{colim}\Bigl(\bigoplus_{i\in\mathbb Z} M(X)_\Q(i)[2i],f\Bigr).
\]
Equivalently, since the category is pseudo-abelian, $\sM(\sT)$ is identified with the image of $f$. Thus there exist morphisms
\[
p:\bigoplus_{i\in\mathbb Z} M(X)_\Q(i)[2i]\to \sM(\sT),\qquad
j:\sM(\sT)\to \bigoplus_{i\in\mathbb Z} M(X)_\Q(i)[2i]
\]
such that
\[
p\circ j=\mathrm{id}_{\sM(\sT)},\qquad j\circ p=f.
\]

To prove that $\sM(\sT)=0$, it is enough to show that for every $m\in\mathbb Z$ and every morphism
\[
r:M(X)_\Q(m)[2m]\to \sM(\sT),
\]
one has $r=0$.

Fix $m\in\mathbb Z$ and a morphism
\[
r:M(X)_\Q(m)[2m]\to \sM(\sT).
\]
Consider the composite
\[
j\circ r:
M(X)_\Q(m)[2m]\longrightarrow \bigoplus_{i\in\mathbb Z} M(X)_\Q(i)[2i].
\]
Since $M(X)_\Q(m)[2m]$ is compact, there exists a finite subset $F\subset \mathbb Z$ such that $j\circ r$ factors through the finite direct sum
\[
\bigoplus_{i\in F} M(X)_\Q(i)[2i].
\]
Let
\[
i_F:\bigoplus_{i\in F} M(X)_\Q(i)[2i]\hookrightarrow \bigoplus_{i\in\mathbb Z} M(X)_\Q(i)[2i]
\]
be the natural inclusion, and let
\[
s_F:\bigoplus_{i\in\mathbb Z} M(X)_\Q(i)[2i]\to \bigoplus_{i\in F} M(X)_\Q(i)[2i]
\]
be the natural projection. Then there exists a morphism
\[
r_F:M(X)_\Q(m)[2m]\to \bigoplus_{i\in F} M(X)_\Q(i)[2i]
\]
such that
\[
j\circ r=i_F\circ r_F,
\qquad\text{where } r_F=s_F\circ j\circ r.
\]
In other words, we have a commutative diagram
\[
\xymatrix{
M(X)_\Q(m)[2m] \ar[r]^-{r} \ar[d]_{r_F}
&
\sM(\sT) \ar[d]^{j}
\\
\displaystyle\bigoplus_{i\in F} M(X)_\Q(i)[2i] \ar[r]^-{i_F}
&
\displaystyle\bigoplus_{i\in\mathbb Z} M(X)_\Q(i)[2i].
}
\]
Consider the endomorphism
\[
\tilde{f}:\bigoplus_{i\in F} M(X)_\Q(i)[2i]
\overset{i_F}{\longrightarrow}
\bigoplus_{i\in\mathbb Z} M(X)_\Q(i)[2i]
\overset{f}{\longrightarrow}
\bigoplus_{i\in\mathbb Z} M(X)_\Q(i)[2i]
\overset{s_F}{\longrightarrow}
\bigoplus_{i\in F} M(X)_\Q(i)[2i].
\]
Then we have
\begin{align*}
\tilde{f}\circ r_F
&= s_F\circ f\circ i_F\circ r_F \\
&= s_F\circ f\circ j\circ r \\
&= s_F\circ j\circ p\circ j\circ r \\
&= s_F\circ j\circ r \\
&= r_F.
\end{align*}

By Theorem~\ref{ladicrealizationthm}, we have
\[
R_l(f)=0.
\]
Hence
\[
R_l(\tilde{f})=0.
\]
Therefore \(\tilde{f}\) is homologically trivial, and in particular numerically trivial. By \cite[Proposition 7.5]{K05}, \(\tilde{f}\) is nilpotent. Thus there exists \(N\geq 1\) such that
\[
\tilde{f}^N=0.
\]
Since \(\tilde{f}\circ r_F=r_F\), it follows that
\[
r_F=\tilde{f}^N\circ r_F=0.
\]
Consequently, \(j\circ r=i_F\circ r_F=0\), and hence
\[
r=p\circ j\circ r=0.
\]
We obtain the claim.
\end{proof}
\section{deformation invariance of phantomness}
For a prime $l$ and a commutative ring $A$ in which $l$ is invertible and which contains a primitive $l^v$-th root of unity, denote by $\beta\in \pi_2 K/l^v(A)$ the Bott element, where $v \geq 1$ if $l>3$, $v \geq 2$ if $l=3$, and $v \geq 4$ if $l=2$.
\begin{prop}\label{deformK1localK}
Let $S$ be a locally Noetherian integral scheme, and let $l$ be a prime number invertible on $S$. Let
\[
f:\sX \to S
\]
be a smooth proper morphism, and let $\sT\subset \perf(\sX)$ be an $S$-linear admissible full subcategory. Then, for any two geometric points $\overline{s},\overline{t}\to S$, there is a non-canonical isomorphism
\[
\pi_j\bigl(K/l^v(\sT_{\overline{s}})[\beta^{-1}]\bigr)\simeq \pi_j\bigl(K/l^v(\sT_{\overline{t}})[\beta^{-1}]\bigr)
\]
for every $j\in \mathbb{Z}$, where $v \geq 1$ if $l>3$, $v \geq 2$ if $l=3$, and $v \geq 4$ if $l=2$.
\end{prop}
\begin{proof}
We may therefore assume that $S=\Spec(A)$ for a Noetherian strictly henselian discrete valuation ring $A$. Let $K$ denote the fraction field of $A$. Since $cd_l(K)=1$, Thomason's theorem \cite[Theorem~2.45 and Theorem~3.1]{T85} yields a strongly convergent spectral sequence
\[
E_2^{i,j}(\sX)=H_{\et}^{i-j}(\sX,\Z/l^v(j))
\Longrightarrow
\pi_j\bigl(K/l^v(\sX)[\beta^{-1}]\bigr).
\]
Likewise, we have strongly convergent spectral sequences
\[
E_2^{i,j}(\sX_{\overline{s}})=H_{\et}^{i-j}(\sX_{\overline{s}},\Z/l^v(j))
\Longrightarrow
\pi_j\bigl(K/l^v(\sX_{\overline{s}})[\beta^{-1}]\bigr),
\]
and
\[
E_2^{i,j}(\sX_{\overline{t}})=H_{\et}^{i-j}(\sX_{\overline{t}},\Z/l^v(j))
\Longrightarrow
\pi_j\bigl(K/l^v(\sX_{\overline{t}})[\beta^{-1}]\bigr).
\]

The natural morphisms $\sX_{\overline{s}}\to \sX$ and 
$\sX_{\overline{t}}\to \sX$ induce morphisms of spectral sequences
\[
\xymatrix@C=4.2em@R=3.2em{
H_{\et}^{i-j}(\sX_{\overline{s}},\Z/l^v(j))
\ar@{=>}[r]
&
\pi_j\bigl(K/l^v(\sX_{\overline{s}})[\beta^{-1}]\bigr)
\\
H_{\et}^{i-j}(\sX,\Z/l^v(j)) \ar@{=>}[r] \ar[u]^{\ol{s}^*} \ar[d]_{\ol{t}^*}
&
\pi_j\bigl(K/l^v(\sX)[\beta^{-1}]\bigr) \ar[u] \ar[d]
\\
H_{\et}^{i-j}(\sX_{\overline{t}},\Z/l^v(j))
\ar@{=>}[r]
&
\pi_j\bigl(K/l^v(\sX_{\overline{t}})[\beta^{-1}]\bigr)
}
\]
By smooth proper base change, the pullback maps
\[
\overline{s}^*:H_{\et}^{i-j}(\sX,\Z/l^v(j))\longrightarrow H_{\et}^{i-j}(\sX_{\overline{s}},\Z/l^v(j))
\]
and
\[
\overline{t}^*:H_{\et}^{i-j}(\sX,\Z/l^v(j))\longrightarrow H_{\et}^{i-j}(\sX_{\overline{t}},\Z/l^v(j))
\]
are isomorphisms. Since the above spectral sequences are strongly convergent, it follows that
\[
\pi_j\bigl(K/l^v(\sX_{\overline{s}})[\beta^{-1}]\bigr)
\simeq
\pi_j\bigl(K/l^v(\sX)[\beta^{-1}]\bigr)
\simeq
\pi_j\bigl(K/l^v(\sX_{\overline{t}})[\beta^{-1}]\bigr).
\]

Let
\[
i:\sT\hookrightarrow \perf(\sX)
\]
be the inclusion, and let
\[
l:\perf(\sX)\to \sT
\]
be its left adjoint. Set
\[
e:=i\circ l:\perf(\sX)\to \perf(\sX).
\]
Then $e$ is an idempotent exact endofunctor. By base change, $e$ induces idempotent exact endofunctors
\[
e_{\overline{s}}:\perf(\sX_{\overline{s}})\to \perf(\sX_{\overline{s}})
\qquad\text{and}\qquad
e_{\overline{t}}:\perf(\sX_{\overline{t}})\to \perf(\sX_{\overline{t}}),
\]
whose essential images are precisely $\sT_{\overline{s}}$ and $\sT_{\overline{t}}$, respectively.

By functoriality of Thomason's spectral sequence with respect to exact functors, the endofunctors
\[
e,\qquad e_{\overline{s}},\qquad e_{\overline{t}}
\]
induce endomorphisms of the above spectral sequences, and these endomorphisms are compatible with the pullback maps $\overline{s}^*$ and $\overline{t}^*$. In particular, the isomorphism
\[
\pi_j\bigl(K/l^v(\sX_{\overline{s}})[\beta^{-1}]\bigr)
\simeq
\pi_j\bigl(K/l^v(\sX_{\overline{t}})[\beta^{-1}]\bigr)
\]
intertwines the endomorphisms induced by $e_{\overline{s}}$ and $e_{\overline{t}}$.

Now, since $\sT_{\overline{s}}\subset \perf(\sX_{\overline{s}})$ and $\sT_{\overline{t}}\subset \perf(\sX_{\overline{t}})$ are admissible, \(K\)-theory splits accordingly, and the image of the endomorphism induced by $e_{\overline{s}}$ (resp. $e_{\overline{t}}$) on
\[
\pi_j\bigl(K/l^v(\sX_{\overline{s}})[\beta^{-1}]\bigr) \text{ (resp. }\pi_j\bigl(K/l^v(\sX_{\overline{t}})[\beta^{-1}]\bigr) )
\]
is naturally identified with
\[
\pi_j\bigl(K/l^v(\sT_{\overline{s}})[\beta^{-1}]\bigr) \text{ (resp. }\pi_j\bigl(K/l^v(\sT_{\overline{t}})[\beta^{-1}]\bigr) \bigr) ).
\]
Therefore the above isomorphism restricts to an isomorphism
\[
\pi_j\bigl(K/l^v(\sT_{\overline{s}})[\beta^{-1}]\bigr)
\simeq
\pi_j\bigl(K/l^v(\sT_{\overline{t}})[\beta^{-1}]\bigr).
\]
This proves the proposition.
\end{proof}
The following corollary follows immediately from \cite[Theorem~A]{BOR20}, Proposition~\ref{deformK1localK}, and Theorem~\ref{criterionphantominsection3}.
\begin{cor} Let $S$ be a Noetherian connected scheme, and let $f:\sX \to S$ be a smooth proper morphism. Assume that the Chow motive of each fiber is Kimura-finite. Suppose that for some point $0\in S$, the fiber $\sX_0$ admits a motivic quasi-phantom category $\sT_0$, and that there exists a prime number $l\in \kappa(0)^*$ such that $L_{K(1,l)}K(\sT_0)\neq 0$. Then there exists an \'etale morphism $\sU \to S$ whose image contains $0$ and such that every fiber of $\sX_{\sU} \to \sU$ admits a motivic quasi-phantom category. Moreover, under the Belmans--Okawa--Ricolfi conjecture \cite[Conjecture 9.9]{BOR20}, the nonvanishing condition $L_{K(1,l)}K(\sT_0)\neq 0$ is unnecessary.
\end{cor}
    
\bibliography{bib}

@article {A14,
    AUTHOR = {Ayoub, Joseph},
     TITLE = {La r\'{e}alisation \'{e}tale et les op\'{e}rations de {G}rothendieck},
   JOURNAL = {Ann. Sci. \'{E}c. Norm. Sup\'{e}r. (4)},
  FJOURNAL = {Annales Scientifiques de l'\'{E}cole Normale Sup\'{e}rieure. Quatri\`eme
              S\'{e}rie},
    VOLUME = {47},
      YEAR = {2014},
    NUMBER = {1},
     PAGES = {1--145},
      ISSN = {0012-9593},
   MRCLASS = {14C15 (14F05 14F20 14F42 18G55)},
  MRNUMBER = {3205601},
MRREVIEWER = {C. A. M. Peters},
       DOI = {10.24033/asens.2210},
       URL = {https://doi.org/10.24033/asens.2210},
}

@incollection {AH61,
    AUTHOR = {Atiyah, M. F. and Hirzebruch, F.},
     TITLE = {Vector bundles and homogeneous spaces.},
 BOOKTITLE = {Proc. {S}ympos. {P}ure {M}ath., {V}ol. {III}},
     PAGES = {7--38},
 PUBLISHER = {, },
      YEAR = {1961},
   MRCLASS = {57.30 (14.52)},
  MRNUMBER = {139181},
MRREVIEWER = {R.\ Bott},
}

@Article{BBKS15,
 Author = {B{\"o}hning, Christian and Graf von Bothmer, Hans-Christian and Katzarkov, Ludmil and Sosna, Pawel},
 Title = {Determinantal {Barlow} surfaces and phantom categories},
 FJournal = {Journal of the European Mathematical Society (JEMS)},
 Journal = {J. Eur. Math. Soc. (JEMS)},
 ISSN = {1435-9855},
 Volume = {17},
 Number = {7},
 Pages = {1569--1592},
 Year = {2015},
 Language = {English},
 DOI = {10.4171/JEMS/539},
 zbMATH = {6483975},
 Zbl = {1323.14014}
}

@Misc{BOR20,
 Author = {Belmans, Pieter and Okawa, Shinnosuke and Ricolfi, Andrea T.},
 Title = {Moduli spaces of semiorthogonal decompositions in families},
 Year = {2020},
 HowPublished = {Preprint, {arXiv}:2002.03303 [math.{AG}] (2020)},
 URL = {https://arxiv.org/abs/2002.03303},
 arXiv = {arXiv:2002.03303}
}

@article {C83,
    AUTHOR = {Connes, Alain},
     TITLE = {Cohomologie cyclique et foncteurs {${\rm Ext}\sp n$}},
   JOURNAL = {C. R. Acad. Sci. Paris S\'{e}r. I Math.},
  FJOURNAL = {Comptes Rendus des S\'{e}ances de l'Acad\'{e}mie des Sciences.
              S\'{e}rie I. Math\'{e}matique},
    VOLUME = {296},
      YEAR = {1983},
    NUMBER = {23},
     PAGES = {953--958},
      ISSN = {0249-6291},
   MRCLASS = {18F25 (19D55 46L80)},
  MRNUMBER = {777584},
}

@book {CD19,
    AUTHOR = {Cisinski, Denis-Charles and D\'{e}glise, Fr\'{e}d\'{e}ric},
     TITLE = {Triangulated categories of mixed motives},
    SERIES = {Springer Monographs in Mathematics},
 PUBLISHER = {Springer, Cham},
      YEAR = {[2019] \copyright 2019},
     PAGES = {xlii+406},
      ISBN = {978-3-030-33241-9; 978-3-030-33242-6},
   MRCLASS = {14F42 (14C15 14C35 18G80 19D55)},
  MRNUMBER = {3971240},
MRREVIEWER = {Igor A. Rapinchuk},
       DOI = {10.1007/978-3-030-33242-6},
       URL = {https://doi.org/10.1007/978-3-030-33242-6},
}

@InCollection{CKP13,
 Author = {Cheltsov, Ivan and Katzarkov, Ludmil and Przyjalkowski, Victor},
 Title = {Birational geometry via moduli spaces},
 BookTitle = {Birational geometry, rational curves, and arithmetic. Based on the symposium ``Geometry over closed fields'', St. John, UK, February 2012},
 ISBN = {978-1-4614-6481-5; 978-1-4614-6482-2},
 Pages = {93--132},
 Year = {2013},
 Publisher = {New York, NY: Springer},
 Language = {English},
 DOI = {10.1007/978-1-4614-6482-2_5},
 zbMATH = {6211437},
 Zbl = {1302.14035}
}

@Article{CT11,
 Author = {Cisinski, Denis-Charles and Tabuada, Gon{\c{c}}alo},
 Title = {Non-connective {{\(K\)}}-theory via universal invariants},
 FJournal = {Compositio Mathematica},
 Journal = {Compos. Math.},
 ISSN = {0010-437X},
 Volume = {147},
 Number = {4},
 Pages = {1281--1320},
 Year = {2011},
 Language = {English},
 DOI = {10.1112/S0010437X11005380},
 zbMATH = {5938398},
 Zbl = {1247.19001}
}

@Article{GO13,
 Author = {Gorchinskiy, Sergey and Orlov, Dmitri},
 Title = {Geometric phantom categories},
 FJournal = {Publications Math{\'e}matiques},
 Journal = {Publ. Math., Inst. Hautes {\'E}tud. Sci.},
 ISSN = {0073-8301},
 Volume = {117},
 Pages = {329--349},
 Year = {2013},
 Language = {English},
 DOI = {10.1007/s10240-013-0050-5},
 zbMATH = {6185247},
 Zbl = {1285.14018}
}

@incollection {FT87,
    AUTHOR = {Fe\u igin, B. L. and Tsygan, B. L.},
     TITLE = {Cyclic homology of algebras with quadratic relations,
              universal enveloping algebras and group algebras},
 BOOKTITLE = {{$K$}-theory, arithmetic and geometry ({M}oscow, 1984--1986)},
    SERIES = {Lecture Notes in Math.},
    VOLUME = {1289},
     PAGES = {210--239},
 PUBLISHER = {Springer, Berlin},
      YEAR = {1987},
      ISBN = {3-540-18571-2},
   MRCLASS = {17B35 (16A03 17B56 18G15 18G55 19D55 58H10)},
  MRNUMBER = {923137},
MRREVIEWER = {Zbigniew\ Marciniak},
       DOI = {10.1007/BFb0078369},
       URL = {https://doi.org/10.1007/BFb0078369},
}

@article {HSO17,
    AUTHOR = {Hoyois, Marc and Kelly, Shane and \O stv\ae r, Paul Arne},
     TITLE = {The motivic {S}teenrod algebra in positive characteristic},
   JOURNAL = {J. Eur. Math. Soc. (JEMS)},
  FJOURNAL = {Journal of the European Mathematical Society (JEMS)},
    VOLUME = {19},
      YEAR = {2017},
    NUMBER = {12},
     PAGES = {3813--3849},
      ISSN = {1435-9855,1435-9863},
   MRCLASS = {14F42 (19E15)},
  MRNUMBER = {3730515},
MRREVIEWER = {Ramdorai\ Sujatha},
       DOI = {10.4171/JEMS/754},
       URL = {https://doi.org/10.4171/JEMS/754},
}

@Article{K05,
 Author = {Kimura, Shun-Ichi},
 Title = {Chow groups are finite dimensional, in some sense},
 FJournal = {Mathematische Annalen},
 Journal = {Math. Ann.},
 ISSN = {0025-5831},
 Volume = {331},
 Number = {1},
 Pages = {173--201},
 Year = {2005},
 Language = {English},
 DOI = {10.1007/s00208-004-0577-3},
 zbMATH = {2132967},
 Zbl = {1067.14006}
}

@misc {L18,
Author = {Jacob Lurie},
Title = {Higher algebra},
Year = {2018},
}

@article {R05,
    AUTHOR = {Riou, Jo\"el},
     TITLE = {Dualit\'e{} de {S}panier-{W}hitehead en g\'eom\'etrie
              alg\'ebrique},
   JOURNAL = {C. R. Math. Acad. Sci. Paris},
  FJOURNAL = {Comptes Rendus Math\'ematique. Acad\'emie des Sciences. Paris},
    VOLUME = {340},
      YEAR = {2005},
    NUMBER = {6},
     PAGES = {431--436},
      ISSN = {1631-073X,1778-3569},
   MRCLASS = {14F35 (55P25)},
  MRNUMBER = {2135324},
MRREVIEWER = {Feng-Wen\ An},
       DOI = {10.1016/j.crma.2005.02.002},
       URL = {https://doi.org/10.1016/j.crma.2005.02.002},
}

@Misc{R12,
 Author = {Robalo, Marco},
 Title = {Noncommutative {Motives} {I}: {A} {Universal} {Characterization} of the {Motivic} {Stable} {Homotopy} {Theory} of {Schemes}},
 Year = {2012},
 HowPublished = {Preprint, {arXiv}:1206.3645 [math.{AG}] (2012)},
 URL = {https://arxiv.org/abs/1206.3645},
 arXiv = {arXiv:1206.3645}
}

@Misc{R13,
 Title={Noncommutative Motives II: K-Theory and Noncommutative Motives},
 Author={Marco Robalo},
 Year={2013},
 HowPublished = {Preprint, {arXiv}:1306.3795[math.{KT}] (2013)},
 URL = {https://arxiv.org/abs/1306.3795},
 arXiv = {arXiv:1306.3795},
}

@article{RO08,
  title={Modules over motivic cohomology},
  author={R{\"o}ndigs, Oliver and {\O}stv{\ae}r, Paul Arne},
  journal={Advances in Mathematics},
  volume={219},
  number={2},
  pages={689--727},
  year={2008},
  publisher={Elsevier}
}

@article {S20,
    AUTHOR = {Sosna, Pawel},
     TITLE = {Some remarks on phantom categories and motives},
   JOURNAL = {Bull. Belg. Math. Soc. Simon Stevin},
  FJOURNAL = {Bulletin of the Belgian Mathematical Society. Simon Stevin},
    VOLUME = {27},
      YEAR = {2020},
    NUMBER = {3},
     PAGES = {337--352},
      ISSN = {1370-1444,2034-1970},
   MRCLASS = {14C15 (14C35 14F08)},
  MRNUMBER = {4146735},
MRREVIEWER = {Joan\ Pons-Llopis},
       DOI = {10.36045/bbms/1599616818},
       URL = {https://doi.org/10.36045/bbms/1599616818},
}

@inproceedings{T85,
  title={Algebraic $ K $-theory and {\'e}tale cohomology},
  author={Thomason, Robert W},
  booktitle={Annales scientifiques de l'{\'E}cole Normale Sup{\'e}rieure},
  volume={18},
  pages={437--552},
  year={1985}
}

@incollection {V00b,
    AUTHOR = {Voevodsky, Vladimir},
     TITLE = {Triangulated categories of motives over a field},
 BOOKTITLE = {Cycles, transfers, and motivic homology theories},
    SERIES = {Ann. of Math. Stud.},
    VOLUME = {143},
     PAGES = {188--238},
 PUBLISHER = {Princeton Univ. Press, Princeton, NJ},
      YEAR = {2000},
   MRCLASS = {14F42 (14C25)},
  MRNUMBER = {1764202},
}

@article {V10,
    AUTHOR = {Voevodsky, Vladimir},
     TITLE = {Cancellation theorem},
   JOURNAL = {Doc. Math.},
  FJOURNAL = {Documenta Mathematica},
      YEAR = {2010},
      NOTE = {Extra volume: Andrei A. Suslin sixtieth birthday},
     PAGES = {671--685},
      ISSN = {1431-0635},
   MRCLASS = {14F42 (19E15)},
  MRNUMBER = {2804268},
MRREVIEWER = {Oliver R\"{o}ndigs},
}
\bibliographystyle{alpha}
\end{document}